\documentclass[12pt]{amsart}
\usepackage{latexsym, url}
\usepackage{amsthm}
\usepackage{amsmath}
\usepackage{amsfonts}
\usepackage{amssymb}
\usepackage[dvips]{graphicx}
\usepackage{xypic}
\input xy
\xyoption{all}

\newtheorem{theo}{Theorem}[section]
\newtheorem{lemma}[theo]{Lemma}

\newtheorem{propo}[theo]{Proposition}
\newtheorem{defi}[theo]{Definition}
\newtheorem{coro}[theo]{Corollary}

\newtheorem{rem}[theo]{Remark}

\newtheorem{exam}[theo]{Example}
\newtheorem{exams}[theo]{Examples}

\newcommand\op{\operatorname{op}}

\newcommand\Set{\operatorname{\bf Set}}

\newcommand\CPO{\operatorname{\bf CPO}}

\newcommand\MGra{\operatorname{\bf MGra}}

\newcommand\Gra{\operatorname{\bf Gra}}

\newcommand\colim{\operatorname{colim}}

\newcommand\ca{\mathcal {A}}
\newcommand\cb{\mathcal {B}}
\newcommand\cc{\mathcal {C}}
\newcommand\cd{\mathcal {D}}

\newcommand\ck{\mathcal {K}}
\newcommand\cl{\mathcal {L}}

\date{March 26, 2018}
 
\begin{document}
\title[Nearly locally presentable categories]
{Nearly locally presentable categories}
\author[L. Positselski and J. Rosick\'{y}]
{L. Positselski and J. Rosick\'{y}}
\thanks{Supported by the Grant Agency of the Czech Republic under
the grant P201/12/G028.  The first-named author's research in Israel
is supported by the ISF grant~\#\,446/15}  
\address{
\newline L. Positselski\newline
Department of Mathematics\newline
Faculty of Natural Sciences, University of Haifa\newline
Mount Carmel, Haifa 31905, Israel -- and --\newline
Institute for Information Transmission Problems\newline
Bolshoy Karetny per.~19, Moscow 127051, Russia -- and -- \newline
Laboratory of Algebraic Geometry\newline
National Research Univerity Higher School of Economics\newline
Usacheva 6, Moscow 119048, Russia\newline
positselski@yandex.ru\newline\newline
J. Rosick\'{y}\newline
Department of Mathematics and Statistics\newline
Masaryk University, Faculty of Sciences\newline
Kotl\'{a}\v{r}sk\'{a} 2, 611 37 Brno, Czech Republic\newline
rosicky@math.muni.cz
}
 
\begin{abstract}
We introduce a new class of categories generalizing locally presentable ones. The distinction does not manifest in the abelian
case and, assuming Vop\v enka's principle, the same happens in the regular case. The category of complete partial orders is the natural example of a nearly locally finitely presentable category which is not locally presentable.  
\end{abstract} 

\maketitle

\section{Introduction}

Locally presentable categories were introduced by P. Gabriel and F. Ulmer in \cite{GU}. A category $\ck$ is locally $\lambda$-presentable if it is cocomplete and has a strong generator consisting of $\lambda$-presentable objects. Here, $\lambda$ is a regular cardinal and an object $A$ is $\lambda$-presentable if its hom-functor $\ck(A,-):\ck\to\Set$ preserves $\lambda$-directed colimits.
A category is locally presentable if it is locally $\lambda$-presentable for some $\lambda$. This concept of presentability
formalizes the usual practice -- for instance, finitely presentable groups are precisely groups given by finitely many generators 
and finitely many relations. Locally presentable categories have many nice properties, in particular they are complete and co-wellpowered.

Gabriel and Ulmer \cite{GU} also showed that one can define locally presentable categories by using just monomorphisms instead all morphisms. They defined $\lambda$-generated objects as those whose hom-functor $\ck(A,-)$ preserves $\lambda$-directed colimits of monomorphisms. Again, this concept formalizes the usual practice -- finitely generated groups are precisely groups admitting a finite set of generators. This leads to locally generated categories, where a cocomplete category $\ck$ is locally $\lambda$-generated if it has a strong generator consisting of $\lambda$-generated objects and every object of $\ck$ has only a set of strong quotients. Since a locally presentable category is co-wellpowered, every locally $\lambda$-presentable category is locally $\lambda$-generated. Conversely, a locally 
$\lambda$-generated category is locally presentable but not necessarily locally $\lambda$-presentable (see \cite{GU} or \cite{AR}). In particular, each locally generated category is co-wellpowered. Under Vop\v enka's principle, we can omit weak co-wellpoweredness in the definition of a locally generated category, but it is still open whether one needs set theory for this (see \cite{AR}, Open Problem 3).

We introduce further weakening of the concept of presentability -- an object $A$ is nearly $\lambda$-presentable if its hom-functor
$\ck(A,-)$ preserves $\lambda$-directed colimits given by expressing a coproduct by its subcoproducts of size $<\lambda$. Any 
$\lambda$-presentable object is nearly $\lambda$-presentable and, if coproduct injections are monomorphisms, any $\lambda$-generated
object is nearly $\lambda$-presentable. In this case, an object $A$ is nearly $\lambda$-presentable if every morphism from $A$ to 
the coproduct $\coprod_{i\in I}K_i$ factorizes through $\coprod_{j\in J}K_i$ where $|J|<\lambda$. This concept is standard
for triangulated categories where the resulting objects are called $\lambda$-small (see \cite{N}). 

We say that a cocomplete category $\ck$ is nearly locally $\lambda$-presentable if it has a strong generator consisting of nearly 
$\lambda$-presentable objects and every object of $\ck$ has only a set of strong quotients. This definition looks quite weak because
$\lambda$-directed colimits used for defining nearly $\lambda$-presentable objects are very special. But, surprisingly, any abelian nearly locally $\lambda$-presentable category is locally presentable. These abelian categories were introduced in \cite{PS} and called locally weakly generated. This is justified by the fact that coproduct injections are monomorphisms there and thus nearly $\lambda$-presentable objects generalize $\lambda$-generated ones. Since weakly locally presentable categories mean something else (see \cite{AR}), we had to change our terminology. We even show that for categories with regular factorizations of morphisms (by a regular epimorphism followed 
by a monomorphism), the fact that nearly locally presentable categories are locally presentable is equivalent to Vop\v enka's principle. Thus we get some artificial examples of nearly locally presentable categories which are not locally presentable under the negation 
of Vop\v enka's principle. A natural example of this, not depending on set theory, is the category $\CPO$ of complete partial orders. It is nearly locally finitely presentable but not locally presentable. This category plays a central role in theoretical computer science, in denotational semantics and domain theory.

\medskip\noindent
{\bf Acknowledgement.} We are very grateful to the referee for suggestions which highly improved our results and presentation.
 In particular, Lemma~\ref{weaklycolimdense} is due to the referee.
 
\section{Nearly presentable objects}

\begin{defi}\label{cdir}
{
\em
$\lambda$-directed colimits $\coprod_{j\in J}K_j\to\coprod_{i\in I}K_i$, where $J$ ranges over all the subsets of $I$ of cardinality less than $\lambda$, will be called \textit{special $\lambda$-directed colimits}.
}
\end{defi}

\begin{defi}\label{bpres}
{
\em
Let $\mathcal K$ be a category with coproducts and $\lambda$ a regular cardinal. An object $A$ of $\mathcal K$ will be called
\textit{nearly} $\lambda$-\textit{presentable} if its hom-functor $\ck(A,-):\ck\to\Set$ preserves special $\lambda$-directed colimits.
}
\end{defi}

\begin{rem}\label{bpres1}
{
\em
(1) This means that $\ck(A,-)$ sends special $\lambda$-directed colimits to $\lambda$-directed colimits and not to special
$\lambda$-directed ones (because $\ck(A,-)$ does not preserve coproducts). 

Explicitly, for every morphism $f:A\to\coprod_{i\in I} K_i$ there is a subset $J$ of $I$ of cardinality less than $\lambda$ such that $f$ factorizes as  
$$
A\xrightarrow{\ g\ } \coprod_{j\in J} K_j\xrightarrow{\ u \ }K_i
$$   
where $u$ is the subcoproduct injection. Moreover, this factorization is essentially unique in the sense that if $f=gu=g'u$ 
then there is a subset $J'$ of $I$ of cardinality $<\lambda$ such that $J\subseteq J'$ and the coproduct injection 
$\coprod_{j\in J} K_j\to\coprod_{j'\in J} K_{j'}$ merges $g$ and $g'$.

(2) If coproduct injections are monomorphisms, the essential uniqueness is automatic. Thus $A$ is nearly $\lambda$-presentable if and only if for every morphism  $f:A\to\coprod_{i\in I}K_i$ there is a subset $J$ of $I$ of cardinality less than $\lambda$ such that $f$ factorizes as $A\to\coprod_{j\in J}K_j\to\coprod_{i\in I}K_i$ where the second morphism is the subcoproduct injection.    

(3) Coproduct injections are very often monomorphisms, for instance in any pointed category. However, in the category of commutative rings, the coproduct is the tensor product and the coproduct injection $\mathbb Z\to\mathbb Z\otimes\mathbb Z/2\cong\mathbb Z/2$ is not a monomorphism.

(4) Any $\lambda$-presentable object is nearly $\lambda$-presentable. We say that $A$ is \textit{nearly presentable} if it is 
nearly $\lambda$-presentable for some $\lambda$.  

(5) An object $K$ is called coproduct-presentable if its hom-functor $\ck(K,-)$ preserves coproducts (see \cite{H}); these objects are also called indecomposable or connected. Any coproduct-presentable object is nearly $\aleph_0$-presentable.
}
\end{rem}

Recall that an epimorphism $f: K\to L$ is strong if each commuting square
$$
\xymatrix@=3pc{
L \ar[r]^{v} & B \\
K \ar [u]^{f} \ar [r]_{u} &
A \ar[u]_{g}
}
$$
such that $g$ is a monomorphism has a diagonal fill-in, i.e., a morphism $t:L\to A$ with $tf=u$ and $gt=v$.  

A colimit of a diagram $D:\cd\to\ck$ is called $\lambda$-small if $\cd$ contains less than $\lambda$ morphisms.

\begin{rem}\label{wgen}
{
\em
(1) A $\lambda$-small colimit of nearly $\lambda$-presentable objects is nearly $\lambda$-presentable. The proof is analogous
to \cite{AR}, 1.3.

(2) If coproduct injections are monomorphisms then any strong quotient of a nearly $\lambda$-presentable object is nearly
$\lambda$-presentable. In fact, let $g:A\to B$ be a strong quotient and $f:B\to\coprod_{i\in I}K_i$. There is $J\subseteq I$
of cardinality less than $\lambda$ such that $fg$ factorizes through $\coprod_{j\in J}K_j$
$$
\xymatrix@=3pc{
B \ar[r]^{f} & \coprod_i K_i \\
A \ar [u]^{g} \ar [r]_{} &
\coprod_j K_j \ar[u]_{u}
}
$$
Since $u$ is a monomorphism, there is a diagonal $h:B\to\coprod_jK_j$ factorizing $f$ through $\coprod_{j\in J}K_j$.
}
\end{rem}

\begin{lemma}\label{up}
If $\lambda_1\leq\lambda_2$ then a nearly $\lambda_1$-presentable object $A$ is nearly $\lambda_2$-presentable.
\end{lemma}
\begin{proof}
Let $A$ be nearly $\lambda_1$-presentable. Then any morphism $f:A\to\coprod_{i\in I}K_i$ factorizes through a subcoproduct
$\coprod_{j\in J}K_j$ where $|J|<\lambda_1\leq\lambda_2$. Assume that we have two factorizations given by 
$f_k:A\to\coprod_{j\in J}K_j$ where $|J|<\lambda_2$. Each of $f_k$ factorizes through $g_k:A\to\coprod_{j\in J'}K_j$, where 
$J'\subseteq J$ and $|J'|<\lambda_1$ and $k=1,2$. There is $J'\subseteq J''\subseteq I$ such that $|J''|<\lambda_1$ and
$\coprod_{J'}K_j\to\coprod_{J''}K_j$ coequalizes $g_1$ and $g_2$. Thus $\coprod_JK_j\to\coprod_{J\cup J''}K_j$ coequalizes $f_1$ and $f_2$.
\end{proof}

\begin{exams}\label{examples}
{
\em
(1) Let $\ck$ be a category where every object is a coproduct of coproduct-presentable objects.
Then coproduct injections are monomorphisms in~$\ck$.
Furthermore, nearly $\lambda$-presentable objects are precisely the $\lambda$-small coproducts of coproduct-presentable objects.
Examples are the categories of sets, graphs, posets, or presheaves of sets. In the category of sets, nearly $\lambda$-presentable objects coincide with $\lambda$-presentable ones. In the categories of graphs and posets, they are precisely the objects having less that $\lambda$ connected components.
Observe that there is a proper class of non-isomorphic nearly $\lambda$-presentable objects there.

(2) More generally, let $\ck$ be a category where every object is a coproduct of nearly $\aleph_0$-presentable objects.
 Then nearly $\lambda$-presentable objects are precisely the $\lambda$-small coproducts of nearly $\aleph_0$-presentable objects.
 The category of vector spaces (over any fixed field) is an example.
 An object of the category of vector spaces is nearly $\lambda$-presentable iff it is $\lambda$-presentable.

(3) An abelian group $A$ is nearly  $\lambda$-presentable iff it is $\lambda$-presentable.

Assume at first that $\lambda$ is an uncountable cardinal.
Then, to say that an abelian group $A$ does not have a set of generators of the cardinality less than $\lambda$, simply means that the cardinality of $A$ is at least~$\lambda$.
Any abelian group is a subgroup of an injective (=~divisible) abelian group (see, e.g., \cite{R}, 3.35(i) and 3.36).
According to the classification of injective abelian groups (or, more generally,
injective modules over a Noetherian commutative ring), an injective abelian
group is a direct sum of indecomposable injectives, which are precisely
the abelian groups $\mathbb Q$ and $\mathbb Q_p/\mathbb Z_p=\mathbb Z[\frac{1}{p}]/\mathbb Z$ (\cite{Mat}, 2.5 and~3.1). All of these are countable.
Thus, to embed an abelian group $A$ of the cardinality at least $\lambda$ into
an injective abelian group $J$, the group $J$ has to be the direct sum of at least
$\lambda$ (indecomposable injective) abelian groups. Hence $A$ is not nearly
$\lambda$-presentable.

Now, assume that $\lambda=\aleph_0$. Let $A$ be an infinitely generated abelian
group. Pick elements $a_1$, $a_2$, $a_3$,~\dots\ in $A$ such that $a_{n+1}$ does not
belong to the subgroup generated by $a_1$,~\dots,~$a_n$ in $A$. Denote this subgroup
by $A_n \subset A$, and the union of $A_n$ over the natural numbers~$n$ by $A_\omega$.
Then we have a natural (injective) morphism of abelian groups $h_\omega:A_\omega\to\bigoplus_{n=1}^\infty A_\omega/A_n$.
For every~$n$, choose an injective abelian group $J_n$ such that $A_\omega/A_n$ is a subgroup in $J_n$.
Then we also have an injective morphism $j_\omega:\bigoplus_{n=1}^\infty A_\omega/A_n\to\bigoplus_{n=1}^\infty J_n$.
Now it is important that the class of injective abelian groups is closed under
direct sums (see~\cite{R}, 3.31). The direct sum of $J_n$ being injective, we can extend the
morphism $j_\omega h_\omega:A_\omega\to\bigoplus_{n=1}^\infty J_n$ to a morphism $h:A\to\bigoplus_{n=1}^\infty J_n$. Since the image of this morphism is not contained in
the direct sum of any finite subset of $J_n$ (as the image of~$h_\omega$ is
not contained in the direct sum of any finite subset of $A_\omega/A_n$),
it follows that $A$ is not nearly finitely presentable.

(4) The same argument as in~(3) applies to modules over any countable
Noetherian commutative ring (in place of the abelian groups).
}
\end{exams}

\section{Nearly locally presentable categories}

Recall that a strong generator is a small full subcategory $\ca$ of $\ck$ such that the functor $E_\ca:\ck\to\Set^{\ca^{\op}}$, 
$EK=\ck(-,K)$, is faithful and conservative (= reflects isomorphisms). A generator $\ca$ of $\ck$ is strong if and only if for each object $K$ and each proper subobject of $K$ there exists a morphism $A\to K$ with $A\in\ca$ which does not factorize through that subobject.

A category is called weakly co-wellpowered if every its object has only a set of strong quotients.  

\begin{defi}\label{blp}
{
\em
A cocomplete category $\ck$ will be called \textit{nearly locally} $\lambda$-\textit{presentable} if it is weakly co-wellpowered and has a strong generator consisting of nearly $\lambda$-presentable objects.  

A category is \textit{nearly locally presentable} if it is nearly locally $\lambda$-presentable for some regular cardinal $\lambda$.
}
\end{defi}

This concept was introduced in \cite{PS} for abelian categories.
Since any abelian category with a generator is co-wellpowered,
weak co-wellpoweredness does not need to be assumed there.
Any locally $\lambda$-presentable category is nearly locally
$\lambda$-presentable (by \cite{AR}, 1.20 and since a locally
presentable category is co-wellpowered~\cite{AR}, 1.58).

\begin{rem}\label{blp1}
{
\em
Let $\ck$ be a nearly locally $\lambda$-presentable category and $\ca$ its strong generator consisting of nearly $\lambda$-presentable objects. Following \ref{bpres1}(1), $E_\ca:\ck\to\Set^{\ca^{\op}}$ sends special $\lambda$-directed colimits to $\lambda$-directed colimits.
}
\end{rem}

We say that a category $\ck$ has monomorphisms \textit{stable under} $\lambda$-\textit{directed colimits} if for every $\lambda$-directed set of subobjects $(K_i)_{i\in I}$ of $K\in\ck$ the induced morphism $\colim_{i\in I}K_i\to K$ is a monomorphism. Following \cite{AR} 1.62,
any locally $\lambda$-presentable category has this property. 

A category has regular factorizations if every morphism can be decomposed as a regular epimorphism followed by a monomorphism.
Any regular category (in particular, any abelian category) has this property (see \cite{B} I.2.3). The following result was proved in \cite{PS} 9.1 for abelian categories and for $\lambda=\aleph_0$.
 
\begin{propo}\label{exact}
Every nearly locally $\lambda$-presentable category with regular factorizations has monomorphisms stable under $\lambda$-directed colimits.
\end{propo}

\begin{proof}
Let $(K_i)_{i\in I}$ be a $\lambda$-directed set of subobjects, $k_i:K_i\to L=\colim_{i\in I}K_i$ a colimit cocone and
$t:L\to K$ the induced morphism. Since $tk_i:K_i\to K$ are monomorphisms, $k_i$ are monomorphisms.
Let $f:\coprod_{i\in I} K_i\to L$ and $p:\coprod_{i\in I} K_i\to K$ be the induced morphisms.
Clearly, $tf=p$.

As we will see below in~\ref{complete}, every nearly locally presentable category is complete; in particular, it has kernel pairs. 
Let $p_1$, $p_2:P\to\coprod_{i\in I}K_i$ be a kernel pair of $p$. It suffices to prove that $fp_1=fp_2$.
In this case, $p_1$, $p_2$ is a kernel pair of $f$ because $fg_1=fg_2$ implies that $pg_1=pg_2$. Since $f$ is a regular epimorphism, it is a coequalizer of $p_1$,~$p_2$.
Since $\ck$ has regular factorizations, $t$ is a monomorphism. Indeed, the regular factorization of $p$ should consist of the coequalizer of $p_1$, $p_2$ followed by the induced morphism.

Let $\ca$ be a strong generator of $\ck$ consisting of nearly $\lambda$-presentable objects. It suffices to prove that $fp_1h=fp_2h$
for any $h:A\to P$, $A\in\ca$. Since $A$ is nearly $\lambda$-presentable, there is $J\subseteq I$ of cardinality less that $\lambda$
such that $p_nh$ factorizes through $\coprod_{j\in J}K_j$ for $n=1$,~$2$. Since $(K_i)_{i\in I}$ is a $\lambda$-directed set, there is $K_m$, $m\in I$ such that $fp_nh$ factorizes through the monomorphism $k_m:K_m\to L$ for $n=1$,~$2$. Let $q_1,q_2:A\to K_m$ be the corresponding factorizations.
Then $tk_mq_1=tk_mq_2$, hence $q_1=q_2$ and thus $fp_1h=fp_2h$.
\end{proof}

Recall that a category is bounded if it has a small dense subcategory
(see \cite{AR}).

\begin{coro}\label{regular}
Every nearly locally presentable category with regular factorizations is bounded.
\end{coro}

\begin{proof}
Let $\ca$ be a strong generator of a nearly locally $\lambda$-presentable category $\ck$ consisting of nearly $\lambda$-presentable objects. Let $\overline{\ca}$ be the closure of $\ca$ under $\lambda$-small colimits and strong quotients. Since $\ck$ is weakly co-wellpowered, $\overline{\ca}$ is small. For an object $K$ of $\ck$ we form the canonical diagram $D$ w.r.t. $\overline{\ca}$ and take its colimit $K^\ast$. Since $\ca$ is a strong generator, it suffices to prove that the induced morphism $t:K^\ast\to K$ is a monomorphism. Then it is an isomorphism because every morphism $f:A\to K$, $A\in\ca$ factorizes through~$t$.

Since $\ck$ has regular factorizations and $\overline{\ca}$ is closed under strong quotients, the subdiagram $D_0$ of $D$ consisting 
of monomorphisms $f:A\to K$, $A\in\overline{\ca}$ is cofinal in $D$. Thus $K^\ast=\colim D_0$ and, following \ref{exact}, $t$ is 
a monomorphism.
\end{proof}

In fact, we have proved that every weakly co-wellpowered category $\ck$ having regular factorizations, a strong generator and monomorphisms closed under $\lambda$-directed colimits is bounded.

\begin{theo}\label{abelian}
Every nearly locally presentable abelian category is locally presentable.
\end{theo}

\begin{proof}
 Consider a nearly locally $\lambda$-presentable abelian category $\ck$.
 According to~\ref{exact}, $\ck$ has $\lambda$-directed unions.
 Following~\cite{Mit} III.1.2 and III.1.9, in any cocomplete abelian
category with monomorphisms closed under directed colimits
the directed colimits are exact (cf.\ \cite{PS} 9.2).
 In the same way we see that in any cocomplete abelian category
with monomorphisms closed under $\lambda$-directed colimits
the $\lambda$-directed colimits are exact (i.e., commute with finite
limits).
 Thus $\ck$ is a cocomplete abelian category with a generator in which
$\lambda$-directed colimits commute with finite limits.
 Following \cite{PR} 2.2, $\ck$ is locally presentable.
\end{proof}

\begin{theo}\label{vp} 
Vop\v enka's principle is equivalent to the fact that every nearly locally presentable category having regular factorizations is locally presentable.
\end{theo}

\begin{proof}
Assuming Vop\v enka's principle, every cocomplete bounded category is locally presentable (see \cite{AR} 6.14). Under the negation 
of Vop\v enka's principle, there is a rigid class of connected graphs $G_i$ in the category $\Gra$ of graphs (see \cite{AR} 6.13). This is a rigid class in the category of multigraphs $\MGra$ which, as a presheaf category on $\cdot\rightrightarrows\cdot$, is regular. Let $\ck$ be the full subcategory of $\MGra$ consisting of all the multigraphs in which every connected component is either the terminal multigraph $1$ or there are no morphisms from $G_i$ into it. This is the modification of \cite{AR} 6.12. Like there, 
$\ck$ is an epireflective subcategory of $\MGra$ and thus it is cocomplete. Moreover, $\ck$ is closed under coproducts in $\MGra$. The graphs $\cdot$ and $\cdot\rightarrow\cdot$ are nearly finitely presentable in $\ck$ because every their morphism into a coproduct uniquely factorizes through a coproduct injection. Since the graphs $\cdot$ and $\cdot\rightarrow\cdot$ form a dense subcategory 
in $\ck$, the category $\ck$ is nearly finitely presentable. But, like in \cite{AR} 6.12, the graph $\cdot$ is not presentable 
in $\ck$. Thus $\ck$ is not locally presentable.

Let us prove that $\ck$ is regular, i.e., that regular epimorphisms are stable under pullback. For this, it suffices to show
that the inclusion $\ck\to\MGra$ preserves regular epimorphisms. Assume that $f:K\to L$ is a regular epimorphism in $\ck$. 
Then $f=f_1\coprod f_2:K_1\coprod K_2\to L_1\coprod L_2$ where $L_2$ is a coproduct of copies of $1$ and there is no morphism
$G_i\to L_1$. Then $f_1:K_1\to L_1$ is a regular epimorphism in $\MGra$. Since the full preimage in $K_2$ of each component of $L_2$ contains an edge, $f_2$ 
is a regular epimorphism in $\MGra$ again. Hence $f$ is a regular epimorphism in $\MGra$.
\end{proof}

\begin{exams}\label{various}
{
\em
(1) Analogously, \cite{AR} 6.36 gives a nearly locally presentable category which, under the negation of Vop\v enka's principle, is not bounded.

(2) \cite{AR} 6.38 gives a cocomplete category $\ck$ having a strong generator consisting of nearly $\lambda$-presentable objects which is not complete. But $\ck$ is not weakly co-wellpowered.
}
\end{exams}

\begin{exam}\label{cpo}
{
\em
Let $\CPO$ be the category of chain-complete posets, i.e., posets where every chain has a join. Morphisms are mappings
preserving joins of chains. Every chain complete poset has the smallest element, and the joins of directed sets and morphisms preserve
them. These posets are also called cpo's and play a central role in theoretical computer science, in denotational semantics and domain theory. The category $\CPO$ is cocomplete (see \cite{M}). The coproduct is just the disjoint union with the least elements of each component identified. Thus every finite cpo is nearly finitely presentable in $\CPO$. Epimorphisms are morphisms $f:A\to B$ where
$f(A)$ is directed join dense in $B$, i.e., every $b\in B$ is a join of a directed set $X\subseteq f(A)$. Thus $|B|\leq 2^{|A|}$, which
implies that $\CPO$ is co-wellpowered. The two-element chain $2$ is a strong generator in $\CPO$. In fact, it is a generator
and for each object $B$ and each proper subobject $A$ of $B$ there exists a morphism $2\to B$ which does not factorize through $A$. 
Hence $\CPO$ is nearly locally finitely presentable. But $\CPO$ is not locally presentable (see \cite{AR} 1.18(5)).
}
\end{exam}

\begin{defi}\label{cunion}
{
\em
Let $\ck$ be a category with coproducts. We say that $\ck$ has monomorphisms \textit{stable under special} $\lambda$-\textit{directed colimits} if for every special $\lambda$-directed colimit $\coprod_{j\in J}K_j\to\coprod_{i\in I}K_i$, every morphism
$\coprod_{i\in I}K_i\to K$ whose compositions with $\coprod_{j\in J}K_j\to\coprod_{i\in I}K_i$ are monomorphisms is a monomorphism.
}
\end{defi}

This definition fits with the stability of monomorphisms under $\lambda$-directed colimits provided that coproduct injections in $\ck$ are monomorphisms.

\begin{propo}\label{dual} 
Let $\ck$ be a locally presentable category such that $\ck^{\op}$ has mono\-mor\-phisms stable under special $\lambda$-directed colimits 
for some regular cardinal $\lambda$. Then $\ck$ is equivalent to a complete lattice.
\end{propo}
\begin{proof}
It follows from \cite{AR} 1.64. In more detail, the proof considers a special $\lambda$-codirected limit 
$p_J:K^I\to K^J$ and $m:\colim D\to K^I$. Using local presentability of $\ck$, $m$ is shown to be a regular monomorphism.
Since $\ck^{\op}$ has monomorphisms stable under special $\lambda$-directed colimits, $m$ is an epimorphism because the compositions 
$p_Jm$ are epimorphisms. Thus $m$ is an isomorphism, which yields the proof.
\end{proof}

\begin{rem}\label{c-acc}
{
\em
\cite{H} calls a category $\ck$ coproduct-accessible if it has coproducts and a set of coproduct-presentable objects such that every object
of $\ck$ is a coproduct of objects from this set. Any cocomplete coproduct-accessible category is nearly locally $\aleph_0$-presentable.
}
\end{rem}
 
\section{Properties of nearly locally presentable categories}

\begin{rem}\label{colimdense}
{
\em
 The following observations will be useful below.

(1) A full subcategory $\ca$ of $\ck$ is called \textit{weakly colimit-dense} if $\ck$ is the smallest full subcategory of $\ck$ containing $\ca$ and closed under colimits. Any weakly colimit-dense full subcategory is a strong generator (see 
\cite{S} 3.7). Conversely, as we will see in~\ref{weaklycolimdense}, in a cocomplete and weakly co-wellpowered category every strong generator is weakly colimit-dense (cf.\ \cite{K} 3.40 or \cite{S} 3.8).

(2) Compact Hausdorff spaces form a cocomplete, regular and weakly co-well\-po\-we\-red category with a strong generator which is not nearly
locally presentable. This follows from \ref{all} and the fact that $1$ is not nearly presentable in compact Hausdorff spaces. 

(3) Recall that a generator $\ca$ of $\ck$ is strong if and only if for each object $K$ and each proper subobject of $K$ there exists a morphism $A\to K$ with $A\in\ca$ which does not factorize through that subobject.
If $\ck$ has equalizers then every small full subcategory $\ca$ of $\ck$ satisfying this condition is a strong generator. Given two morphisms $f,g:K\to L$, it suffices to apply this condition to the equalizer of $f$ and~$g$.

In a cocomplete category, a generator $\ca$ is strong iff every object is an extremal quotient of a coproduct of objects from $\ca$
(see \cite{AR} 0.6). Recall that an epimorphism $f:K\to L$ is extremal if any monomorphism $L'\to L$ through which $f$ factorizes
is an isomorphism.
}
\end{rem}

\begin{lemma}\label{weaklycolimdense}
Let $\ca$ be a strong generator in a cocomplete, weakly co-wellpovered category~$\ck$.
Then $\ca$ is weakly colimit-dense in~$\ck$.
\end{lemma}

\begin{proof}
Let $K$ be an object of $\ck$. Following \ref{colimdense}(3), there is an extremal epimorphism
$e_0:K_0\to K$ where $K_0$ is a coproduct of objects of $\ca$. If $e_0$ is a monomorphism, it is an isomorphism. Thus $K$ belongs
to the closure of $\ca$ under colimits.

If $e_0$ is not a monomorphism, there are distinct morphisms $f_1$, $f_2:M\to K_0$ such that
$e_0f_1=e_0f_2$. Since $\ca$ is a generator of $\ck$, we can assume that $M\in\ca$.  Let $e_{01}:K_0\to K_1$ be the coequalizer of $f_1$, $f_2$ and $e_1:K_1\to K$ the induced morphism. If $e_1$ is a monomomorphism, then it is an isomorphism and $K$ belongs to the iterated closure of $\ca$ under colimits.
If $e_1$ is not an monomorphism, we repeat the procedure. In this way,
we get the chain 
$$K_0 \xrightarrow{\ e_{01}\ } K_1 \xrightarrow{\ e_{12}\ } \dots
$$
consisting of strong epimorphisms where in limit steps we take colimits. Since $\ck$ is weakly co-wellpowered, the construction stops and we get that $K$ belongs to the iterated colimit closure of $\ca$.
\end{proof}

\begin{propo}\label{complete}
Every nearly locally presentable category is complete.
\end{propo}

\begin{proof}
Let $\ck$ be a nearly locally presentable category. 
Following \ref{weaklycolimdense}, $\ck$ has a weakly colimit-dense small subcategory.
By \cite{AHR} Theorem 4, any cocomplete weakly co-wellpowered category
with a weakly colimit-dense small subcategory is complete.
\end{proof}
 
\begin{lemma}\label{all}
Let $\ck$ be a nearly locally presentable category such that coproduct injections are monomorphisms. Then every object of $\ck$ is nearly presentable.
\end{lemma}

\begin{proof}
Let $\ca$ be a strong generator of $\ck$ consisting of nearly presentable objects. Following~\ref{complete}, $\ck$ has pullbacks.
Every object of $\ck$ is then a strong quotient
of a coproduct of objects from $\ca$ (see \cite{AR} 0.6 and 0.5). The result follows from \ref{wgen}. 
\end{proof}

\begin{lemma}\label{colwg}
Let $\cl$ be a cocomplete well-powered and weakly co-wellpowered category with (strong epimorphism, monomorphism)-factorization. 
Let $\ca$ be a set of nearly presentable objects in $\cl$.
Let $\ck$ consist of strong quotients of coproducts of objects from $\ca$.
The $\ck$ is nearly locally presentable.
\end{lemma}

\begin{proof}
Following \cite{AHS} 16.8, $\ck$ is a coreflective full subcategory of $\cl$.
Hence it is cocomplete and weakly co-wellpowered.
$\ca$ is a strong generator in $\ck$ consisting of nearly presentable objects.
Thus $\ck$ is nearly locally presentable.
\end{proof}

\begin{propo}\label{refl}
Let $\ck$ be a reflective subcategory of a locally $\lambda$-presentable category such that the inclusion $G:\ck\to\cl$ is conservative
and sends special $\lambda$-directed colimits to $\lambda$-directed colimits. If $\ck$ is complete, cocomplete and weakly co-wellpowered then it is nearly locally $\lambda$-presentable.
\end{propo}

\begin{proof}
Let $F$ be a left adjoint to $G$ and consider a $\lambda$-presentable object $L$ in $\cl$. Since $\ck(FL,-)\cong\cl(L,G-)$
and $G$ sends special $\lambda$-directed colimits to $\lambda$-directed colimits, $FL$ is nearly $\lambda$-presentable in $\ck$.  
We prove that the objects $FL$, where $L$ ranges over $\lambda$-presentable objects in $\cl$, form a strong generator of $\ck$.
We use the argument from the end of the proof of 2.9 in \cite{AR1}, which we repeat. Following \ref{colimdense}(3), it suffices
to show that for every proper subobject $m:K'\to K$ in $\ck$ there exists a morphism from some $FL$ to $K$, where $L$ is 
$\lambda$-presentable in $\cl$, not factorizing through $m$. We know that $Gm$ is a monomorphism but not an isomorphism
because $G$ is conservative. Since $\cl$ is locally $\lambda$-presentable, there exists a morphism $p:L\to GK$, $L$ $\lambda$-presentable
in $\cl$, that does not factorize through $Gm$. The the corresponding $\tilde{p}:FL\to K$ does not factorize through $m$.
\end{proof}

\begin{coro}\label{functor}
Let $\ck$ be a nearly locally $\lambda$-presentable category and $\cc$ be a small category. Then the functor category $\ck^\cc$ is nearly locally $\lambda$-presentable.
\end{coro}

\begin{proof}
$\ck^\cc$ is complete and cocomplete (with limits and colimits calculated pointwise).
It is easy to see that $\varphi:P\to Q$ is a strong epimorphism in $\ck^\cc$ if and only if $\varphi_C:PC\to QC$ is a strong epimorphism on $\ck$ for each $C$ in $\cc$. Thus $\ck^\cc$ is weakly co-wellpowered. Let $\ca$ be a strong generator of $\ck$
consisting of nearly $\lambda$-presentable objects. Following \ref{blp1}, $E_\ca:\ck\to\Set^{\ca^{\op}}$ makes $\ck$ a reflective
subcategory of a locally $\lambda$-presentable category $\cl=\Set^{\ca^{\op}}$ with the conservative inclusion functor sending special 
$\lambda$-directed colimits to $\lambda$-directed colimits. Thus $\ck^\cc$ is a reflective subcategory of a locally $\lambda$-presentable category  $\cl^\cc$ with the conservative inclusion functor sending special $\lambda$-directed colimits to $\lambda$-directed colimits.
Following \ref{refl}, $\ck^\cc$ is nearly locally $\lambda$-presentable.
\end{proof}

\begin{theo}\label{char}
Let $\ck$ be a cocomplete category with regular factorizations in which coproduct injections are monomorphisms. Then $\ck$ is nearly
locally $\lambda$-presentable if and only if it is a full reflective subcategory of a presheaf category such that the inclusion preserves
special $\lambda$-directed colimits.
\end{theo}

\begin{proof}
Since $\ck$ has regular factorizations, regular and strong epimorphisms coincide (see \cite{AR} 0.5). Let $\ck$ be a full reflective subcategory of a presheaf category such that the inclusion preserves special $\lambda$-directed colimits. Then $\ck$ is complete
and every object has only a set of regular quotients. Following \ref{refl}, $\ck$ is nearly locally $\lambda$-presentable. 

Let $\ck$ be nearly locally $\lambda$-presentable and $\ca$ be a strong generator consisting of nearly $\lambda$-presentable objects.
Let $\cb$ be the closure of $\ca$ under $\lambda$-small colimits and strong quotients. Following \ref{regular}, $\cb$ is dense in $\ck$,
and, following \ref{wgen}, every object of $\cb$ is nearly $\lambda$-presentable. Thus the functor $E_{\cb}:\ck\to \Set^{\cb^{\op}}$
preserves special $\lambda$-directed colimits and makes $\ck$ a full reflective subcategory of the presheaf category (see \cite{AR} 1.27).
\end{proof}

\begin{propo}\label{up1}
Let $\lambda_1\leq\lambda_2$ be regular cardinals. Then any nearly locally $\lambda_1$-presentable category is nearly locally 
$\lambda_2$-presentable.
\end{propo}

\begin{proof}
 Follows from~\ref{up}.
\end{proof}

\begin{lemma}\label{mono}
Every nearly locally $\lambda$-presentable category has monomorphisms stable under special $\lambda$-directed colimits.
\end{lemma}

\begin{proof}
Let $f:\coprod_IK_i\to K$ be a morphism whose compositions
$f_J:\coprod_JK_j\to K$ with the subcoproduct injections
$\coprod_JK_j\to\coprod_IK_i$ are monomorphisms for all
$J\subseteq I$ of cardinality $<\lambda$.
It suffices to show that $u=v$ for any 
$u,v:A\to\coprod_IK_i$ such that $fu=fv$ and $A$ is nearly $\lambda$-presentable.
Since $u$ and $v$ factorize through a subcoproduct injection
$\coprod_JK_j\to\coprod_IK_i$, we have $fu=f_Ju'$ and $fv=f_Jv'$
for $u',v':A\to\coprod_JK_j$.
Since $f_J$ is a monomorphism, $u'=v'$ and thus $u=v$.
\end{proof}

\begin{theo}\label{dual1}
Let $\ck$ be a locally presentable category such that $\ck^{\op}$ is nearly locally presentable. Then $\ck$ is equivalent to a complete lattice.
\end{theo}

\begin{proof}
It follows from \ref{mono} and \ref{dual}.
\end{proof}




\begin{thebibliography}{99}  

\itemsep=2pt
 
\bibitem{AHR} J. Ad\' amek, H. Herrlich and J. Reiterman, {\em Cocompleteness almost implies completeness}, Proc. Conf. "Categorical
Topology, Prague 1988", World Scientific 1989, 246-256.

\bibitem{AHS} J. Ad\' amek, H. Herrlich and G. E. Strecker, {\em Abstract and Concrete Categories}, Wiley 1990.

\bibitem{AR} J. Ad\'{a}mek and J. Rosick\'{y}, {\em Locally Presentable and Accessible Categories}, Cambridge Univ.
Press 1994.

\bibitem{AR1} J. Ad\'{a}mek and J. Rosick\'{y}, {\em On reflective subcategories of locally presentable categories}, 
Th. Appl. Cat. 30 (2015), 1306-1318.

\bibitem{B} M. Barr, {\em Exact categories and categories of sheaves}, in Lecture Notes in Mathematics 236, 1-120,
Springer-Verlag 1971.

 
 
\bibitem{GU} P. Gabriel and F. Ulmer, {\em Lokal Pr\" asentierbare Kategories}, Lect. Notes Math. 221, Springer-Verlag 1971.

\bibitem{H} H. Hu {\em Dualities for accessible categories}, In: Category Theory 1991, CMS Conf. Proc. 13, Amer. Math. Soc. 1992, 211-242.



\bibitem{K} G. M. Kelly, {\em The Basic Concepts of Enriched Category Theory}, Cambridge Univ. Press 1982.


\bibitem{M} G. Markowsky, {\em Categories of chain-complete posets}, Th. Comp. Sci. 4 (1977) 125-135.

\bibitem{Mat} E. Matlis, {\em Injective modules over Noetherian rings}, Pacific J. Math. 8 (1958), 511-528.
 
\bibitem{Mit} B. Mitchell, {\em Theory of Categories}, Academic Press 1965.

\bibitem{N} A. Neeman, {\em Triangulated categories}, Princeton Univ. Press 2001. 

\bibitem{PR} L. Positselski and J. Rosick\'{y}, {\em Covers, envelopes, and cotorsion theories in locally presentable abelian categories and contramodule categories}, J. Alg. 483 (2017), 83-128.

\bibitem{PS} L. Positselski and J. \v S\v tov\'\i{}\v cek, {\em The tilting-cotilting correspondence}, arXiv:1710.02230.

\bibitem{R} J. J. Rotman, {\em An Introduction to Homological Algebra}, Second Edition, Springer 2009.

\bibitem{S} M. A. Shulman, {\em Generators and colimit closures}, \newline
http://home.sandiego.edu/~shulman/papers/generators.pdf


\end{thebibliography}
\end{document}